\documentclass[11pt, reqno]{amsart}
\usepackage{amsmath, amsthm, amscd, amsfonts, amssymb, mathtools, color, hyperref}

\usepackage[dvipsnames]{xcolor}

\newtheorem{theorem}{Theorem}[section]
\newtheorem{lemma}[theorem]{Lemma}
\newtheorem{proposition}[theorem]{Proposition}
\newtheorem{corollary}[theorem]{Corollary}
\newtheorem{observation}[theorem]{Observation}

\theoremstyle{definition}
\newtheorem{definition}[theorem]{Definition}
\newtheorem{example}[theorem]{Example}
\theoremstyle{remark}

\numberwithin{equation}{section}

\newcommand{\T}{\mathbb{T}}
\newcommand{\Ti}{\mathbb{T}[i]}
\newcommand{\op}{\oplus}
\newcommand{\ot}{\otimes}

\newcommand{\eps}{\varepsilon}
\newcommand{\PhiSet}{\Phi}

\newcommand{\E}{\mathcal{E}}

\title[Lie's Theorem]{Lie's Theorem for Supertropical Algebra}

\author{ Himadri Mukherjee$^{1,\ast}$}
\address{$^{1}$ Department of Mathematics, BITS Pilani K. K. Birla Goa Campus, Goa, India}
\email{himadrim@goa.bits-pilani.ac.in}
\author[Askar]{Askar Ali M$^{2}$}
\address{$^{2}$ Azim Premji University, Bhopal, India}
\email{askar.m@apu.edu.in}

\thanks{$^\ast$Corresponding author}

\subjclass[2020]{Primary: 15A80, 17B45, 22E45, 17B30   Secondary: 37C25 }
\keywords{Max-plus algebra; Lie algebra; Lie's theorem; Nilpotent matrix}

\begin{document}

\noindent

\begin{abstract}
The aim of this paper is to prove a version of Lie's theorem for the supertropical algebra. 
\end{abstract}

\maketitle
 
\section{Introduction}

Tropical (max-plus) mathematics replaces the classical operations of addition and multiplication by
\[
a \oplus b := \max\{a,b\},\qquad a \otimes b := a+b,
\]
typically on the set $\T:=\mathbb{R}\cup\{-\infty\}$ (with additive identity $\varepsilon=-\infty$ and multiplicative identity $0$).
This ``max--plus'' algebra provides an effective linear-algebraic language for problems that are nonlinear over fields but become linear over $\T$, with applications ranging from discrete-event systems and scheduling to optimization and dynamical systems
(see, for example, \cite{BCOQ-book,Butkovic-book,HOvdW-book,AkianBapatGaubert-HLA}).
In this setting, $n\times n$ matrices over $\T$ form a semiring under tropical matrix addition (entrywise $\oplus$) and tropical matrix multiplication (using $\oplus$ and $\otimes$),
and the resulting ``linear'' theory is tightly intertwined with weighted directed graphs:
to a matrix $A=(a_{ij})$ one associates a digraph with an edge $i\to j$ whenever $a_{ij}\neq \varepsilon$.
Many structural properties of tropical matrices admit graph-theoretic interpretations, and spectral notions are governed by directed cycles and their weights \cite{BCOQ-book,Butkovic-book}.

A basic limitation of $\T$ is that $\oplus$ is idempotent, so ``ties'' (multiple maximizers) are collapsed, and there are no additive inverses.
Supertropical algebra, initiated by Izhakian and collaborators, enriches tropical algebra by adjoining \emph{ghost} elements that record such degeneracies and restore a more robust algebraic structure for polynomials, matrices, and linear algebra \cite{IR-Supertropical-algebra,IR-Supertropical-matrix,IR-Supertropical-linear}.
Motivated by these developments, in this paper we work with a concrete supertropical extension $\T[i]$ obtained by adjoining a central symbol $i$ with $i^2=0$ and distinguishing a ghost (``zero'') ideal $\Phi\subset \T[i]$ (defined precisely in \S2).
This framework is well-suited for formulating matrix-theoretic notions such as nilpotency and for exploiting the close connection between tropical products and directed walks in graphs.

The purpose of the present note is to prove an analogue of \emph{Lie's theorem} in this supertropical setting.
Recall that, over an algebraically closed field, Lie's classical theorem asserts that every finite-dimensional solvable Lie algebra of matrices can be simultaneously upper triangularized; in particular, a nilpotent Lie algebra of matrices is simultaneously strictly upper triangular in a suitable basis \cite{Humphreys-Lie}.
In tropical and supertropical algebra one cannot form the usual commutator $AB-BA$, so one must choose a substitute for the Lie bracket.
Here we consider the bilinear bracket on $M_n(\T[i])$ given by
\[
[A,B]\;:=\;A\otimes B\;\oplus\;B\otimes A,
\]
and we call a $\T[i]$-linear subspace $G\subseteq M_n(\T[i])$ a (supertropical) Lie algebra if it is closed under this operation.
We focus on the nilpotent case, defined via the lower central series in the natural way.

Our main result shows that nilpotent supertropical Lie algebras admit a simultaneous strict upper-triangular form after a permutation change of coordinates (see Theorem~\ref{thm:lie-T[i]}). More precisely, we prove that if $G\subseteq M_n(\T[i])$ is nilpotent, then there exists a permutation matrix $P$ such that
\[
P^{-1}\otimes G\otimes P \;\subseteq\; \mathcal{U}_n,
\]
where $\mathcal{U}_n$ denotes the set of strictly upper triangular matrices (equivalently, every element of $G$ becomes strictly upper triangular after relabeling the basis).
This can be viewed as a tropical/supertropical counterpart of the strict upper-triangularization conclusion for nilpotent Lie algebras in the classical theory \cite{Humphreys-Lie}. As a corollary, the same conclusion holds for nilpotent Lie algebras of max--plus matrices over $\T$ (Corollary~\ref{cor:lie-T}).

\section{Preliminaries}

\begin{definition}[Max--plus semifield]
Let $\T=\mathbb{R}\cup\{-\infty\}$ with
\[
a\oplus b=\max\{a,b\},\qquad a\otimes b=a+b
\]
(where $+$ is the usual addition on $\mathbb{R}$). The additive identity is $\eps=-\infty$ and
the multiplicative identity is 0.
\end{definition}

\begin{definition} 
The super-tropical extension of $\T$ is defined as follows. Adjoin a central symbol $i$ satisfying $i^2=0$. Set
\[
\T[i]=\{\,a+ib:\ a,b\in \T\,\},
\]
where the display ``$a+ib$'' denotes pairing (not tropical addition).
For $x=a+ib$ and $y=c+id$, define
\[
x\oplus y=(a\oplus c)+i(b\oplus d),
\]
\[
x\otimes y=\bigl((a\otimes c)\oplus(b\otimes d)\bigr)+
i\bigl((b\otimes c)\oplus(a\otimes d)\bigr).
\]
The additive zero is $\eps=\eps+i\eps$.
\end{definition}

\begin{observation}\label{obs:nozerodiv}
If $x,y\in \T[i]$ and $x\otimes y=\eps$, then $x=\eps$ or $y=\eps$.
In particular, if $x\neq \eps$ and $y\neq \eps$, then $x\otimes y\neq \eps$.
\end{observation}

\begin{definition}
Dfine the set $\Phi=\{\,a+ia:\ a\in\T\,\}\subset\Ti$. We regard $\PhiSet$ as the ``zero" set. Note that $\eps=\eps+i\eps\in\PhiSet$.
\end{definition}

\begin{observation}\label{Ob:phi-absorb} We have the following facts for the above mentioned zero set.
\begin{itemize}
    \item[(a)] If $x\in\PhiSet$ and $y\in\Ti$, then $x\ot y\in\PhiSet$ and $y\ot x\in\PhiSet$. Moreover, if $x,y\in\PhiSet$, then $x\op y\in\PhiSet$.
    \item[(b)] If $x,y\in\Ti$ and $x\ot y=\eps$, then $x=\eps$ or $y=\eps$.
\end{itemize}

\end{observation}

\begin{definition}
    Let $A \in M_n(\T[i])$. We say $A$ is nilpotent, if $A^k = \E$, for some $k \in \mathbb{N}$, where $\E$ denotes the all-$\eps$ `zero' matrix. 
\end{definition}

\begin{definition}\label{def:digraph}
For $A=(a_{ij})\in M_n(\T[i])$, the directed graph $G_A$ has vertex set
$V=\{1,\dots,n\}$ and a directed edge $i\to j$ iff $a_{ij}\neq \varepsilon$.
\end{definition}

\begin{theorem}\label{thm:nilpDAG}
Let $A\in M_n(\T[i])$. Then $A$ is nilpotent if and only if $G_A$ is a directed acyclic graph (DAG), including no self-loops. Further, spectrum of $A$ is $ \{\varepsilon\}$.
\end{theorem}

\begin{proof}
($\Rightarrow$) Suppose $A^k=\E$ for some $k$. If $G_A$ had a directed cycle
$v_0\to v_1\to\cdots\to v_{m-1}\to v_0$, then for every $r\ge 1$ there is a directed walk
of length $mr$ from $v_0$ to itself obtained by repeating the cycle $r$ times.
In the tropical matrix product, the $(v_0,v_0)$-entry of $A^{mr}$ is an $\oplus$-sum of
$\otimes$-products along all directed walks of length $mr$; in particular it contains the
product of the $mr$ edge-labels along the repeated cycle.
Since each cycle edge corresponds to a non-$\eps$ entry of $A$, Observation~\ref{obs:nozerodiv}
implies that this product is non-$\eps$, hence $(A^{mr})_{v_0v_0}\neq \eps$.
This contradicts $A^{k}=\E$. Thus $G_A$ is acyclic.

($\Leftarrow$) Conversely, assume $G_A$ is a DAG. Since $G_A$ is finite, it has a maximal
directed path length $L$. Then there is no directed walk of length $L+1$.
But $(A^{L+1})_{pq}$ is precisely the $\oplus$-sum of $\otimes$-weights of all directed walks
of length $L+1$ from $p$ to $q$; hence $(A^{L+1})_{pq}=\eps$ for all $p,q$.
Therefore $A^{L+1}=\E$, so $A$ is nilpotent.
\end{proof}

\begin{definition}\cite{Deograph}
Let $G = (V,E)$ be a finite directed graph, with $|V| = n$ and edge set $E$. A \textit{topological ordering} 
$\preceq$ of $G$ is a bijection $\pi:V \rightarrow \{1,\dots, n\}$ (an ordering of the vertices $v_1,\dots,v_n$) 
such that for every edge $v_i \to v_j$ in $E$, with distinct vertices $v_i, v_j$, we have $\pi(v_i) < \pi(v_j)$ (each pair of edge points  
moves forward in the ordering).
\end{definition}

\begin{proposition}\label{lem:toposort}[(Theorem $14-4$ of \cite{Deograph})]
A finite directed graph admits a topological ordering if and only if it is acyclic. In that case, there exists a bijection
$\ell:V\to\{1,\dots,n\}$ such that every edge $u\to v$ satisfies $\ell(u)<\ell(v)$.
\end{proposition}

\begin{definition}
    For $A, B \in M_n(\T[i])$, define the Lie bracket as $[A,B] = AB \oplus BA$. A linear subspace $\mathcal{G} \subset M_n(\T[i])$ is called a Lie algebra, if for every $A,B \in \mathcal{G}$, $[A,B] \in \mathcal{G}$. Further, a linear subspace $\mathcal{I} \subset \mathcal{G}$ is called an ideal, if for any $X \in \mathcal{G}$, and $A \in \mathcal{I}$, $[X,A] \in \mathcal{I}$. 
\end{definition}
\begin{definition}[Lower central series and nilpotent Lie algebra]
Define
\[
D^{0}(\mathcal{G})=\mathcal{G},\qquad
D^{k}(\mathcal{G})=[\,\mathcal{G},\ D^{k-1}(\mathcal{G})\,]\quad(k\ge 1).
\]
We say $\mathcal{G}$ is \emph{nilpotent} if $D^{m}(\mathcal{G})=\{\E\}$ for some $m$.
\end{definition}
\textbf{Remark:} In the super-tropical Lie algebra framework as well, nilpotency implies solvability. By letting the derived series  $\mathcal{G}^{(0)}=\mathcal{G}$, $\mathcal{G}^{(k+1)}=[\mathcal{G}^{(k)}, \mathcal{G}^{(k)}]$, the bracket monotonicity (if $U_1\subseteq U_2$ and $V_1\subseteq V_2$ are $\T[i]$-linear subspaces, then $[U_1,V_1]\ \subseteq\ [U_2,V_2]$) gives 
$\mathcal{G}^{(k)}\subseteq \mathcal{D}_k(\mathcal{G})$ for all $k$. Hence if $\mathcal{D}_m(\mathcal{G})=\{\E\}$, then $\mathcal{G}^{(m)}=\{\E\}$. 
For context, the same implication holds in the classical theory of Lie algebras and groups.

\begin{example}
\begin{itemize}
    \item[(a)] Let $A \in M_n(\T[i])$ be nilpotent. Then the polynomial algebra $\T[i]_0[A]$ is a nilpotent algebra.
    \item[(b)] Define $\mathcal{U}_n \subset M_n(\T[i])$ to be the class of strictly upper triangular matrices. Then $\mathcal{U}_n$ is a nilpotent algebra.
\end{itemize}
\end{example}

\section{Lie's Theorem for Supertropical Algebra}

\begin{lemma}\label{lem:elemnilp}
Let $\mathcal{G}\subseteq M_n(\T[i])$ be a nilpotent Lie algebra.
Then every $A\in\mathcal{G}$ is a nilpotent matrix.
\end{lemma}

\begin{proof}
Fix $A\in\mathcal{G}$. Note that
\[
[A,A]=A\otimes A\oplus A\otimes A=A^{2},
\]
since $\oplus$ is idempotent. Hence $A^{2}\in D^{1}(\mathcal{G})$.

Inductively, assume $A^{k}\in D^{k-1}(\mathcal{G})$. Then
\[
[A,A^{k}]
=A\otimes A^{k}\oplus A^{k}\otimes A
=A^{k+1}\oplus A^{k+1}
=A^{k+1},
\]
so $A^{k+1}\in D^{k}(\mathcal{G})$.

If $D^{m}(\mathcal{G})=\{\E\}$, then $A^{m+1}\in D^{m}(\mathcal{G})$ forces $A^{m+1}=\E$,
hence $A$ is nilpotent.
\end{proof}

\begin{lemma}\label{lem:twoway}
Let $\mathcal{G}\subseteq M_n(\T[i])$ be a nilpotent Lie algebra, and let $A,B\in\mathcal{G}$.
Then there do not exist vertices $v\neq w$ such that there is a directed path from $v$ to $w$
in $G_A$ and a directed path from $w$ to $v$ in $G_B$.
\end{lemma}

\begin{proof}
Assume, for contradiction, that such $v,w$ exist. Let $l$ be the length of a directed path
from $v$ to $w$ in $G_A$ and $m$ the length of a directed path from $w$ to $v$ in $G_B$.
Then $(A^{l})_{vw}\neq\eps$ and $(B^{m})_{wv}\neq\eps$ by the path-expansion property of tropical
matrix products and Observation~\ref{obs:nozerodiv}.

Consider $C=[A^{l},B^{m}]=A^{l}\otimes B^{m}\oplus B^{m}\otimes A^{l}$.
In the $(v,v)$-entry of $A^{l}\otimes B^{m}$, one term is
\[
(A^{l})_{vw}\otimes (B^{m})_{wv}\neq \eps,
\]
again by Observation~\ref{obs:nozerodiv}. Hence $C_{vv}\neq \eps$.
Therefore $G_C$ has a self-loop at $v$, so $C$ is not nilpotent.
But $A^{l},B^{m}\in\mathcal{G}$ and $\mathcal{G}$ is closed under brackets, hence $C\in\mathcal{G}$,
contradicting Lemma~\ref{lem:elemnilp}. Thus no such pair $v,w$ exists.
\end{proof}

 As a result of above Lemmas \ref{lem:elemnilp}, \ref{lem:twoway}, we can define a partial order in any nilpotent algebra $\mathcal{G}$. We say $A \geq B$, if every path in $B$ is a sub-path in $A$. Equivalently (since paths of length $1$ are the edges), $E(G_B)\subseteq E(G_A)$. This guarantees a dominant element in $\mathcal{G}$.

 \begin{theorem}[Lie's Theorem for supertropical algebra] \label{thm:lie-T[i]}
Let $\mathcal{G}\subseteq M_n(\T[i])$ be a nilpotent Lie algebra.
Then there exists a permutation matrix $P$ such that
\[
P^{-1}\otimes \mathcal{G}\otimes P\ \subseteq\ \mathcal{U}_n.
\]
Equivalently, $\mathcal{G}$ is simultaneously conjugate to a Lie subalgebra of strictly upper
triangular matrices.
\end{theorem}

\begin{proof}
First let us construct the dominant element in $\mathcal{G}$. For each ordered pair $(i,j)$ with $i\neq j$ such that there exists some $A\in\mathcal{G}$ with
$a_{ij}\neq\eps$, choose one witness matrix $A^{(i,j)}\in\mathcal{G}$ having that entry non-$\eps$.
Define
\[
C=\bigoplus_{(i,j)} A^{(i,j)}\in\mathcal{G},
\]
since $\mathcal{G}$ is closed under $\oplus$.
By construction,
\[
c_{ij}\neq\eps \quad\Longleftrightarrow\quad
\exists\,A\in\mathcal{G}\ \text{with}\ a_{ij}\neq\eps.
\]
Hence every edge that appears in some $G_A$ ($A\in\mathcal{G}$) appears in $G_C$, i.e.
\[
E(G_A)\subseteq E(G_C)\qquad\text{for all }A\in\mathcal{G}.
\]

Since $\mathcal{G}$ is nilpotent, Lemma~\ref{lem:elemnilp} implies $C$ is nilpotent.
Therefore, $G_C$ is a DAG by Theorem~\ref{thm:nilpDAG}.
By Proposition~\ref{lem:toposort}, there exists a labeling $\ell:\{1,\dots,n\}\to\{1,\dots,n\}$
such that each edge $u\to v$ of $G_C$ satisfies $\ell(u)<\ell(v)$.
Let $P$ be the permutation matrix implementing this relabeling.

Fix $A\in\mathcal{G}$. If $p\to q$ is an edge of $G_A$, then $p\to q$ is an edge of $G_C$, hence $\ell(p)<\ell(q)$. Thus, in the relabeled matrix $P^{-1}\otimes A\otimes P$, every non-$\eps$ entry lies strictly
above the diagonal, i.e.\ $P^{-1}\otimes A\otimes P\in \mathcal{U}_n$.
Since $A$ was arbitrary, $P^{-1}\otimes\mathcal{G}\otimes P\subseteq \mathcal{U}_n$.
\end{proof}

 \begin{corollary}[Lie's Theorem for max-plus algebra]\label{cor:lie-T}
     Let $\mathcal{G} \subset M_n(\T)$ be a nilpotent Lie algebra, then there is a labeling such that $\mathcal{G} \subset \mathcal{U}_n$. In other words, there exists a $P \in Gl_n(\T)$, such that $P^{-1}\mathcal{G}P \subset \mathcal{U}_n$.
 \end{corollary}
The proof follows the same argument as the proof of the above Theorem \ref{thm:lie-T[i]}.

\bibliographystyle{amsplain}

\end{document}